\documentclass[10pt,a4paper]{article}
\usepackage{amsmath}
\usepackage{amsthm}
\usepackage{amsfonts}
\usepackage{amssymb}
\usepackage{stmaryrd}
\usepackage{latexsym}

\usepackage{multirow}
\usepackage{array}
\newcolumntype{P}[1]{>{\centering\arraybackslash}p{#1}}
\usepackage{longtable}

\newtheorem{theorem}{Theorem}

\newtheorem{corollary}{Corollary}

\newtheorem{proposition}{Proposition}

\theoremstyle{definition}

\newtheorem{remark}{Remark}
\newtheorem{example}{Example} 

\begin{document}

 \tolerance2500

\title{\Large{\textbf{On  groupoids  with Bol-Moufang type  identities}}}
\author{\normalsize {Grigorii Horosh, Victor Shcherbacov, Alexandru Tcachenco, Tatiana Yatsko}
}

 \maketitle

\begin{abstract}
We present results about  groupoids of small order  with Bol-Moufang type identities both classical and non-classical  which are listed in  \cite{Cote, Fenyves_1}.

\medskip


\medskip

\end{abstract}

\bigskip

\section{Introduction}

Binary groupoid $(G, \ast)$ is  a non-empty set $G$ together with a binary operation \lq\lq $\ast$\rq\rq \, which is defined on the set $G$.

An identity based on a single binary operation is of Bol-Moufang type if \lq\lq both sides consist of the same three different letters taken in the same order but one of them occurs twice on each side\rq\rq \cite{Fenyves_1}. We use list of 60 Bol-Moufang type identities given in \cite{JAIEOLA_2}.  

There exist other (more general) definitions of Bol-Moufang type identities and, therefore, other lists and classifications of such identities \cite{AKHTAR, Cote}.
An identity based on a single binary operation is of generalized  Bol-Moufang type if \lq\lq both sides consist of the same three different letters  but one of them occurs twice on each side\rq\rq \cite{AKHTAR, Cote}.
In this paper we use both classifications.

Identities from Fenyves's list  we shall name here classical Bol-Moufang type identities. Generalized Bol-Moufang type identities  we shall name here as non-classical type identities. It is clear that any classical type identity is and of non-classical type, but inverse is not true.

 Quasigroups and loops, in which a Bol-Moufang type identity is  true, are  central and classical objects of Quasigroup Theory. We recall, works of R.~Moufang, G.~Bol, R.~H.~Bruck, V.~D.~Belousov, K. Kunen,  S.~Gagola III  and of many other mathematicians are devoted to the study of quasigroups and loops with Bol-Moufang type identities \cite{MUFANG, BRUCK_60_1, VD, Fenyves_1, KUNEN_96,  Gagola_10, Gagola_11, Ph_2005, Cote, AKHTAR}.

We continue the study of groupoids with  Bol-Moufang type identities \cite{NOVIKOV_08,  2017_Scerb, CHErnov, CHErnov_2018, Cernov_2018_2, HSTY_Iasi, HSTY_Chisin}.


\section{Results}
\subsection{Some results on groupoids of order two}

It is clear that there exist 16 groupoids of order 2 and there exist $n^{n^2}$ of groupoids of order $n$.

We list isomorphic pairs of groupoids of order two. If a groupoid does not have a pair, then this groupoid has automorphism group of order two.

Below quadruple $22  \,\,  12$ means groupoid of order 2 with the following Cayley table :

\[
\begin{array}{r|rr}
\ast & 1 & 2 \\
\hline
    1 & 2 & 2  \\
    2 & 1 & 2  \\
    \end{array}
\]
and so on. In such record groupoid is commutative if and only two elements, the second and the third, of a quadruple are  equal.

Groupoid $(G, \cdot)$ is isomorphic to groupoid $(G, \circ)$ if there exists a permutation $\alpha$ of symmetric group $S_G$ such that $x\circ y = \alpha^{-1} (\alpha x \cdot \alpha y) $ for all $x, y \in G$.

Groupoid $(G, \cdot)$ is anti-isomorphic to groupoid $(G, \circ)$ if there exists a permutation $\alpha$ of symmetric group $S_G$ such that $x\circ y = \alpha^{-1} (\alpha y \cdot \alpha x) $ for all $x, y \in G$.

\begin{remark}
If groupoid $(G, \cdot)$ is anti-isomorphic to groupoid $(G, \circ)$, then  groupoid $(G, \circ)$ is anti-isomorphic to groupoid $(G, \cdot)$. Really $x\cdot y = \alpha (\alpha^{-1} y \circ \alpha^{-1} x) $ for all $x, y \in G$.
\end{remark}

\begin{remark}
In commutative groupoid $(G, \cdot)$ any anti-isomorphism coincides with isomorphism.
\end{remark}

It is easy to check that the following propositions are fulfilled.

\begin{proposition}
Only the following groupoids of order two  are isomorphic in pairs:

$11 \,   11$ and $22 \,\,  22$;

$11 \,\,  12$ and    $12 \,\, 22$;

$11 \,\,  21$ and $21 \,\,  22$;

$11 \,\,  22$;

$12 \,\,  11$ and $22 \,\,  12$;

$12 \,\,  12$;

$12 \,\,  21$ and $21 \,\,   12$;

$21 \,\,  11$ and $22 \,\,  21$;

$21 \,\, 21$;

$22 \,\,  11$.
\end{proposition}

\begin{proposition}
Only the following groupoids of order two  are anti-isomorphic in pairs:

$11 \,   21$ and $22 \,\,  12$;

$21 \,   22$ and $12 \,\,  11$;

$11 \,   22$ and $12 \,\,  12$;

$21 \,  \, 21$ and $22 \,\,  11$.
\end{proposition}

\begin{proposition} \label{proposit_Grig}
Only the following groupoids of order two  are  isomorphic or  anti-isomorphic:

$11 \,   11$ and $22 \,\,  22$;

$11 \,\,  12$ and    $12 \,\, 22$;

$11 \,\,  21$ and $21 \,\,  22$ and $22 \,\,  12$ and $12 \,\,  11$;

$11 \,   22$ and $12 \,\,  12$;

$12 \,\,  21$ and $21 \,\,   12$;

$21 \,\,  11$ and $22 \,\,  21$;

$21 \,  \, 21$ and $22 \,\,  11$.
\end{proposition}

\begin{corollary} \label{corollaryGrig}
The following groupoids of order two  are non   isomorphic and non  anti-isomorphic in pairs:
$\underline{ 11 \,   11}$;
$\underline{11 \,\,  12}$;
$11 \,\,  21$;
$\underline{11 \,   22}$;
$\underline{12 \,\,  21}$;
$21 \,\,  11$;
$21 \,  \, 21$.
\end{corollary}

Using the list of groupoids which is presented in Proposition \ref{proposit_Grig} we can compose other lists of groupoids for Corollary \ref{corollaryGrig}. For example, instead of groupoid  $\underline{ 11 \,   11}$ we can write groupoid $\underline{ 22 \,   22}$ and so on.

In the list presented in Corollary \ref{corollaryGrig} semigroups of order two are underlined  \cite{WIKI_44}.

\subsection{$(12)$-parastrophes of identities}

We recall, $(12)$-parastroph of groupoid $(G, \cdot)$ is a groupoid $(G, \ast)$ in which  operation \lq\lq $\ast$ \rq\rq \, is obtained by the following rule:
\begin{equation} \label{HSTY_1}
x\ast y = y\cdot x.
\end{equation}

It is clear that   for any groupoid $(G, \cdot)$   there exists its $(12)$-parastroph groupoid $(G, \ast)$.

Cayley table of groupoid $(G, \ast)$ is a mirror image of the Cayley table of groupoid $(G, \cdot)$ relative to main diagonal.
Notice, for any binary quasigroup there exist  five its parastrophes \cite{VD, HOP, 2017_Scerb} more.

Suppose that an identity $F$ is true in groupoid $(G, \cdot)$.
Then we can obtain $(12)$-parastrophic identity $F^{\ast}$ of the identity $F$ replacing the operation \lq\lq $\cdot$\rq\rq \, with  the operation  \lq\lq $\ast$\rq\rq \, and changing the order of variables using rule (\ref{HSTY_1}).
\begin{remark}
In quasigroup case,  similarly to $(12)$-parastrophe identity other parastrophe identities can be  defined. See \cite{Scer_Shcherb_2} for details.
\end{remark}

It is clear that an identity $F$ is true in groupoid $(G, \cdot)$ if and only if in groupoid $(Q, \ast)$ identity $F^{\ast}$ is true.

\begin{proposition}
The number of groupoids of a finite fixed order  in which the identity $F$ is true coincides with the number of groupoids in which  the identity  $F^{\ast}$ is true.
\end{proposition}

\begin{example}\label{Example_1_HY} \cite{HSTY_Iasi}.
\textit{ Way 1.} We find $(12)$-parastroph of the Bol-Moufang type identity $F_1$: $xy\cdot zx = (xy\cdot z)x$.

We have $(x\ast z) \ast (y\ast x) = x \ast (z\ast (y\ast x))$. After renaming of variables ($y\leftrightarrow z$) and operation ($\ast \rightarrow \cdot $) we obtain the following Bol-Moufang type identity $F_3$: $xy \cdot zx = x  (y\cdot zx)$.

Therefore $(F_1)^{\ast} = F_3$. It is true and vice versa $(F_3)^{\ast} = F_1$.

\textit{ Way 2.}
We recall, left translation of a groupoid $(G, \cdot)$ is defined as follow: $L_a x = a\cdot x$ for all $x\in G$;
right translation of a groupoid $(G, \cdot)$ is defined similarly: $R_a x = x\cdot a$ for all $x\in G$ and a fixed element $a\in G$.

Then we can re-write identity  $F_1$ in the following form: $L_x y\cdot R_x z = R_x(L_x y\cdot z)$.

There exists the following connections between left and right translations of a groupoid $(G, \cdot)$ and its $(12)$-parastrophe    \cite{HOP, 2017_Scerb}:

\begin{equation} \label{HSTY_2}
L^{\ast}_a= R^{\cdot}_a, R^{\ast}_a= L^{\cdot}_a.
\end{equation}

Further using rules (\ref{HSTY_1}) and (\ref{HSTY_2}) we have $L_x z \cdot R_x y  = L_x(z\cdot R_x y)$, $xz\cdot yx = x(z\cdot yx)$. After renaming of variables ($y\leftrightarrow z$) we obtain the following Bol-Moufang type identity $F_3$: $xy \cdot zx = x  (y\cdot zx)$, i.e., $(F_1)^{\ast} = F_3$.
\end{example}

\begin{theorem} \label{12-parastrophes_1}
For classical Bol-Moufang type identities over groupoids the following equalities are true:

$(F_1)^{\ast} = F_3$, $(F_2)^{\ast} = F_4$, $(F_5)^{\ast} = F_{10}$, $(F_6)^{\ast} = F_6$, $(F_7)^{\ast} = F_8$,
$(F_9)^{\ast} = F_9$,   $(F_{11})^{\ast} = F_{24}$, $(F_{12})^{\ast} = F_{23}$,    $(F_{13})^{\ast} = F_{22}$,
$(F_{14})^{\ast} = F_{21}$, $(F_{15})^{\ast} = F_{30}$, $(F_{16})^{\ast} = F_{29}$, $(F_{17})^{\ast} = F_{27}$, $(F_{18})^{\ast} = F_{28}$,
$(F_{19})^{\ast} = F_{26}$, $(F_{20})^{\ast} = F_{25}$,  $(F_{31})^{\ast} = F_{34}$, $(F_{32})^{\ast} = F_{33}$, $(F_{35})^{\ast} = F_{40}$,
 $(F_{36})^{\ast} = F_{39}$, $(F_{37})^{\ast} = F_{37}$, $(F_{38})^{\ast} = F_{38}$, $(F_{41})^{\ast} = F_{53}$, $(F_{42})^{\ast} = F_{54}$,
 $(F_{43})^{\ast} = F_{51}$,$(F_{44})^{\ast} = F_{52}$, $(F_{45})^{\ast} = F_{60}$, $(F_{46})^{\ast} = F_{56}$, $(F_{47})^{\ast} = F_{58}$,
 $(F_{48})^{\ast} = F_{57}$, $(F_{49})^{\ast} = F_{59}$, $(F_{50})^{\ast} = F_{55}$.
\end{theorem}

For quasigroups, analogue of Theorem \ref{12-parastrophes_1} is given in \cite{KUNEN_96}.

\begin{proposition}
Any from the following groupoids $11  \: 11$, $ 22 \;  22$,  $11 \:  12$, $ 12 \;  22$, $11 \;  22$, $12 \;  12$ satisfies any from the identities $F_1$--$F_{60}$.
\end{proposition}
\begin{proof}
It is possible to use direct calculations.
\end{proof}

\subsection{Number of groupoids}

We count number of groupoids  of order two  with classical Bol-Moufang type identities given in \cite{JAIEOLA_2} including and number of non-isomorphic and number of non-isomorphic and   non-anti-isomorphic groupoids of order 2. See Table 1. Notice, in some places Table \ref{tab:table6} coincides with corresponding table from \cite{HSTY_Iasi}.

Table \ref{tab:table6} is organised as follows: in the first column it is given name of identity in Fen'vesh list; in the second it is given abbreviation of this identity, if this identity  has a name; in the third it is given identity; in the fourth column it is indicated the number of groupoids of order 2 with corresponding identity;   in the fifth column -- of non-isomorphic groupoids; and, in the sixth  column --  non-isomorphic and non-anti-isomorphic groupoids with corresponding identity.

\renewcommand{\arraystretch}{1.3}

\begin{longtable}{|p{0.88cm}|p{1.62cm}|p{3.2cm}|p{0.5cm}|p{0.5cm}| p{0.4cm}|p{0.6cm}}
\caption{Number of groupoids  of order 2 with classical  Bol-Moufang  identities}
\label{tab:table6}\\
\hline
\textbf{Na-} & \textbf{Abb.} & \textbf{Ident.} & \textbf{2} & \textbf{n.-} & \textbf{n.-is.,}\\
\textbf{me} & \textbf{} & \textbf{} & \textbf{} & \textbf{is.} & \textbf{an.}\\
\hline
$F_1$ &   &  $xy\cdot zx = (xy\cdot z)x$ & 10 & 6 & 5\\
\hline
 $F_2$ &   &  $xy\cdot zx = (x\cdot yz)x$ & 9 & 6 &5 \\
\hline
$F_3$ &   & $xy\cdot zx = x(y\cdot zx)$ & 10& 6& 5\\
\hline
$F_4$ & middle\,\, Mouf.  & $xy\cdot zx = x(yz\cdot x)$ & 9 & 6 & 5  \\
\hline
$F_5$ &   & $(xy \cdot z)x = (x \cdot y z)x$ & 11 & 7 & 6 \\
\hline
$F_6$ &  extra \,\, ident. & $(xy \cdot z)x = x (y  \cdot z x)$ & 10 & 7 & 5 \\
\hline
$F_7$ &   & $(xy \cdot z)x = x (y z \cdot x)$ & 9 & 6 & 5 \\
\hline
$F_8$ &   & $(x \cdot y z)x = x (y  \cdot z x)$ & 9 & 6 & 5 \\
\hline
$F_9$ &   & $(x \cdot y z)x = x (y   z \cdot x)$ & 10 & 6& 5 \\
\hline
$F_{10}$ &   & $x(y\cdot zx) = x(y z \cdot x)$ & 11 & 7 & 6\\
\hline
$F_{11}$ &   & $xy\cdot xz = (xy \cdot x)z$ & 8 & 5 &4\\
\hline
$F_{12}$ &   & $xy\cdot xz = (x\cdot y x)z$ & 9 & 7 & 6 \\
\hline
$F_{13}$ &  extra\,\, ident.  & $ xy\cdot xz = x(y x\cdot z) $ & 9 & 6 &5\\
\hline
$F_{14}$  &   & $xy\cdot xz = x(y \cdot x z)$ & 10 & 6 &5\\
\hline
$F_{15}$  &   & $(xy \cdot x)z = (x\cdot y  x)z$ & 11 & 7 &6\\
\hline
$F_{16}$  &   & $(xy \cdot x)z = x(y  x\cdot z)$ & 11 & 7 & 6\\
\hline
$F_{17}$  & left \,\, Mouf.   & $(xy \cdot x)z = x(y \cdot  x z)$ & 10 & 7 & 5 \\
\hline
$F_{18}$  &   & $(x\cdot y  x)z = x(y   x \cdot z)$ & 8& 5 & 4 \\
\hline
$F_{19}$  &  left \,\ Bol  & $(x\cdot y  x)z = x(y \cdot  x  z)$ & 9 & 6 & 5 \\
\hline
$F_{20}$  &   & $x(y  x \cdot z) = x(y \cdot  x  z)$ & 9 & 6 & 5\\
\hline
$F_{21}$   &  & $yx \cdot zx = (yx\cdot z)x$  & 10 & 6 & 5\\
\hline
$F_{22}$   & extra \,\, ident.   & $yx \cdot zx = (y \cdot x z)x$ &  9& 6   &5\\
\hline
$F_{23}$   &   & $yx \cdot zx = y ( x z \cdot x)$ & 9 & 6 & 5\\
\hline
$F_{24}$  &   & $yx \cdot zx = y ( x \cdot z  x)$ & 8 & 5 & 4 \\
\hline
$F_{25}$    &    & $(yx \cdot z)x = (y \cdot x z)x$ & 9 & 6 & 5 \\
\hline
$F_{26}$   &  right \,\, Bol  & $(yx \cdot z)x = y (x z \cdot x)$ & 9& 6 &5 \\
\hline
$F_{27}$    &  right \,\, Mouf.  & $(yx \cdot z)x = y (x\cdot  z  x)$ & 10 & 7 & 5\\
\hline
$F_{28}$    &   & $(y \cdot xz)x = y(xz\cdot x)$ & 8 & 5 & 4\\
\hline
$F_{29}$    &   & $(y\cdot xz)x  = y(x \cdot zx)$ & 11 & 7 & 6 \\
\hline
$F_{30}$   &   & $y(xz \cdot x) = y(x \cdot zx)$  & 11 & 7 & 6\\
\hline
$F_{31}$    &   & $yx \cdot xz = (yx \cdot x)z$ & 8 & 5 & 4 \\
\hline
$F_{32}$   &   & $yx \cdot xz = (y\cdot x x)z$ & 9 & 6 & 5\\
\hline
$F_{33}$   &   & $yx \cdot xz = y( x x \cdot z)$ & 9 & 6 & 5\\
\hline
$F_{34}$   &   & $yx \cdot xz = y( x \cdot  x z)$  & 8 & 5 & 4\\
\hline
$F_{35}$   &   & $(yx \cdot x) z = (y\cdot xx)z$  & 9 & 6 & 5\\
\hline
$F_{36}$   & RC \,\, ident. & $(yx \cdot x) z = y(xx \cdot z)$ & 9 & 6 & 5 \\
\hline
$F_{37}$   & C\,\, ident.  & $(yx \cdot x)z = y(x\cdot xz)$ &  10 & 7 & 5\\
\hline
$F_{38}$  &  & $(y \cdot xx) z = y(xx \cdot z)$ & 8 & 5 & 4\\
\hline
$F_{39}$   & LC\,\, ident.  & $(y \cdot xx) z = y(x\cdot xz)$ & 9 & 6 & 5\\
\hline
$F_{40}$   &  & $y (xx\cdot z)  = y (x\cdot xz)$ & 9 & 6 & 5\\
\hline
$F_{41}$   & LC \,\, ident.  & $xx \cdot yz = (x \cdot xy)z$ & 9 & 6 & 5 \\
\hline
$F_{42}$   &  & $xx\cdot yz = (xx \cdot y)$ z & 12 & 7 & 5 \\
\hline
$F_{43}$   &  & $xx\cdot yz  = x(x \cdot yz)$ & 8 & 5 & 4 \\
\hline
$F_{44}$   &  & $xx\cdot yz  = x(x y \cdot z)$ & 9 & 6 & 5 \\
\hline
$F_{45}$   &  & $(x \cdot xy)z = (xx \cdot y)z$ & 9 & 6 & 5\\
\hline
$F_{46}$   & LC \,\, ident. & $(x\cdot xy)z = x(x \cdot yz)$ & 11 & 7 & 6\\
\hline
$F_{47}$   &  & $(x\cdot xy)z = x(xy \cdot z)$ & 8 & 5 & 4 \\
\hline
$F_{48}$   & LC\,\, ident.  & $(xx \cdot y)z = x(x \cdot yz)$ & 10& 7 & 5\\
\hline
$F_{49}$   &  & $(xx \cdot y)z = x(xy \cdot z)$ & 9 & 6 & 5\\
\hline
$F_{50}$   &  & $x(x\cdot y z) = x(xy\cdot z)$ & 11 & 7 & 6\\
\hline
$F_{51}$   &  & $yz\cdot xx  = (yz\cdot x)x$ & 8 & 5 & 4\\
\hline
$F_{52}$   &  & $yz\cdot xx  = (y\cdot z x)x$ & 9 & 6 & 5\\
\hline
$F_{53}$   & RC\,\, ident.  & $yz\cdot xx  = y(z x \cdot x)$ & 9 & 6 & 5\\
\hline
$F_{54}$   &  & $yz\cdot xx  = y(z \cdot x  x)$ & 12 & 7 & 6 \\
\hline
$F_{55}$   & & $(yz \cdot x) x = (y \cdot zx)x$  & 11 & 7 & 6 \\
\hline
$F_{56}$   & RC\,\, ident.  & $(yz \cdot x)x  = y (zx\cdot x)$ & 11& 7 & 6 \\
\hline
$F_{57}$   & RC\,\, ident. & $(yz\cdot x)x  = y (z \cdot xx)$ & 10 & 7 & 6 \\
\hline
$F_{58}$   &  & $(y \cdot zx)x = y (zx\cdot x)$ & 8 & 5 & 4 \\
\hline
$F_{59}$   &  & $(y \cdot zx)x = y (z\cdot x x)$ & 9 & 6 & 5 \\
\hline
$F_{60}$   &  & $y(zx \cdot x) = y(z \cdot xx)$ & 9 & 6 & 5 \\
\hline
\end{longtable}

A new algorithm was developed and the corresponding program was written for generating the  groupoids of small (2, 3, and 4) orders with generalized  Bol-Moufang identities.

Notice, number of groupoids  of order $3$ with mentioned in table identities are also given in \cite{Cernov_2018_2}.

Identities Left Bol and Right Bol, LC- and RC-, LN- and RN-, L2 and L3, M1 and M3, M2 and M4, T1 and T3, T4 and T5, are  $(12)$-parastrophic identitities. Therefore the numbers of  groupoids of fixed order with these $(12)$-parastrophic identitities coincide.

\begin{longtable}{|m{2.5cm}|c|c|c|c|c|}
\caption{Number of groupoids  of order 2, 3 and 4 with Bol-Moufang  identities.}\\
\hline
\textbf{Name} & \textbf{Abbr.} & \textbf{Ident.} &
\textbf{\#2}  & \textbf{\#3}  & \textbf{\#4} \\
\hline
\endfirsthead
\multicolumn{6}{c}%
{\tablename\ \thetable\ -- \textit{Continued from previous page}} \\
\hline
\textbf{Name} & \textbf{Abbr.} & \textbf{Ident.} &
\textbf{\#2} & \textbf{\#3}  & \textbf{\#4} \\
\hline
\endhead
\hline \multicolumn{6}{r}{\textit{Continued on next page}} \\
\endfoot
\hline
\endlastfoot

Extra &
EL &
$ x(y(zx)) = ((xy)z)x $ &
10 &
239 &
18744 \\

Moufang &
ML &
$ (xy)(zx) = (x(yz))x $ &
9 &
196 &
25113 \\

Left Bol &
LB &
$ x(y(xz)) = (x(yx))z $ &
9 &
215 &
22875 \\

Right  Bol &
RB &
$ y((xz)x) = ((yx)z)x $ &
9 &
215 &
22875 \\

C-loops &
CL &
$ y(x(xz)) = ((yx)x)z $ &
10 &
209 &
26583 \\

LC-loops &
LC &
$ (xx)(yz) = (x(xy))z $ &
9 &
220 &
26583 \\

RC-loops &
RC &
$ y((zx)x) = (yz)(xx) $ &
9 &
220 &
26583 \\

Middle Nuclear Square &
MN &
$ y((xx)z) = (y(xx))z $ &
8 &
350 &
122328 \\

Right Nuclear Square &
RN &
$ y(z(xx)) = (yz)(xx) $ &
12 &
932 &
2753064 \\

Left Nuclear Square &
LN &
$ ((xx)y)z = (xx)(yz) $ &
12 &
932 &
2753064 \\

Comm. Moufang &
CM &
$ (xy)(xz) = (xx)(zy) $ &
8 &
297 &
111640 \\


Comm. C-loop &
CC &
$ (y(xy))z = x(y(yz)) $ &
8 &
169 &
12598 \\

Comm. Alternative &
CA &
$ ((xx)y)z = z(x(yx)) $ &
6 &
110 &
10416 \\

Comm. Nuclear square &
CN &
$ ((xx)y)z = (xx)(zy) $ &
9 &
472 &
1321661 \\

Comm. loops &
CP &
$ ((yx)x)z = z(x(yx)) $ &
8 &
744 &
1078744 \\

Cheban, 1 &
C1 &
$ x((xy)z) = (yx)(xz) $ &
8 &
219 &
19846 \\

Cheban, 2 &
C2 &
$ x((xy)z) = (y(zx))x $ &
6 &
153 &
12382 \\

Lonely, I &
L1 &
$ (x(xy))z = y((zx)x) $ &
6 &
117 &
6076 \\

Cheban, I, Dual &
CD &
$ (yx)(xz) = (y(zx))x $ &
8 &
219 &
19846 \\

Lonely, II &
L2 &
$ (x(xy))z = y((xx)z) $ &
7 &
157 &
11489 \\

Lonely, III &
L3 &
$ (y(xx))z = y((zx)x) $ &
7 &
157 &
11489 \\

Mate, I &
M1 &
$ (x(xy))z = ((yz)x)x $ &
6 &
111 &
11188 \\

Mate, II &
M2 &
$ (y(xx))z = ((yz)x)x $ &
7 &
196 &
26785 \\

Mate, III &
M3 &
$ x(x(yz)) = y((zx)x)  $ &
6 &
111 &
11188 \\

Mate, IV  &
M4 &
$ x(x(yz)) = y((xx)z)  $ &
7 &
196 &
26785 \\

Triad, I &
T1 &
$ (xx)(yz) = y(z(xx)) $ &
6 &
162 &
67152 \\

Triad, II &
T2 &
$ ((xx)y)z = y(z(xx)) $ &
6 &
180 &
53832 \\

Triad, III &
T3 &
$ ((xx)y)z = (yz)(xx) $ &
6 &
162 &
67152 \\

Triad, IV &
T4 &
$ ((xx)y)z = ((yz)x)x $ &
6 &
132 &
42456 \\

Triad, V &
T5 &
$ x(x(yz)) = y(z(xx)) $ &
6 &
132 &
42456 \\

Triad, VI  &
T6 &
$ (xx)(yz) = (yz)(xx) $ &
8 &
1419 &
9356968 \\

Triad, VII  &
T7 &
$ ((xx)y)z = ((yx)x)z $ &
12 &
428 &
2914658 \\

Triad, VIII  &
T8 &
$ (xx)(yz) = y((zx)x) $ &
6 &
120 &
11580 \\

Triad, IX  &
T9 &
$ (x(xy))z = y(z(xx)) $ &
6 &
102 &
6192 \\

Frute  &
FR &
$ (x(xy))z = (y(zx))x $ &
6 &
129 &
16600 \\

Crazy Loop  &
CR &
$ (x(xy))z = (yx)(xz) $ &
7 &
136 &
12545 \\

Krypton  &
KR &
$ ((xx)y)z = (x(yz))x $ &
9 &
268 &
93227 \\

\end{longtable}


\vspace{2mm}
\begin{center}
\begin{parbox}{118mm}{\footnotesize
Grigorii Horosh$^{1}$, Victor Shcherbacov$^{2}$,
Alexandr Tcachenco$^{3}$, Tatiana Yatsko$^{4}$
\vspace{3mm}

\noindent $^{1}$Ph.D. Student/Institute of Mathematics and Computer Science of Moldova

\noindent Email: grigorii.horos@math.md

\vspace{3mm}

\noindent $^{2}$Principal Researcher/Institute of Mathematics and Computer Science of Moldova

\noindent Email: victor.scerbacov@math.md

\vspace{3mm}

\noindent $^{3}$Master Student /Institute of Mathematics and Computer Science of Moldova

\noindent Email: atkacheno405@gmail.com

\vspace{3mm}

\noindent $^{4}$Master Student/Shevchenko Transnistria State University

\noindent Email: yaczkot@bk.ru

}

\end{parbox}
\end{center}

\end{document}